\documentclass[twoside,a4paper,12pt,centertags]{amsart}
\usepackage{amsmath,amssymb,verbatim,vmargin}
\usepackage[bookmarks=true]{hyperref}   

\theoremstyle{plain}
\newtheorem{thm}{Theorem}[section]

\newtheorem{prop}[thm]{Proposition}

\theoremstyle{definition}
\newtheorem{defn}[thm]{Definition}

\newtheorem{ex}[thm]{Example}
\newtheorem{que}[thm]{Question}

\mathchardef\semic="303B
\newcommand{\R}{{\mathbf R}}
\newcommand{\C}{{\mathbf C}}

\newcommand{\mH}{{\mathcal H}}

\newcommand{\mL}{{\mathcal L}}
\newcommand{\mF}{{\mathcal F}}

\DeclareMathOperator{\re}{Re}

\newcommand{\sett}[2]{ \{ #1 \, \semic \, #2 \} }
\newcommand{\bigsett}[2]{ \Big\{ #1 \semic #2 \Big\} }

\newcommand{\nul}{\textsf{N}}
\newcommand{\ran}{\textsf{R}}
\newcommand{\dom}{\textsf{D}}

\newcommand{\clos}[1]{\overline{#1}}

\newcommand{\sgn}{\text{{\rm sgn}}}
\newcommand{\barint}{\mbox{$ave \int$}}
\newcommand{\divv}{{\text{{\rm div}}}}
\newcommand{\curl}{{\text{{\rm curl}}}}

\newcommand{\dd}[2]{\frac{d #1}{d #2}}

\newcommand{\ta}{{\scriptscriptstyle \parallel}}
\newcommand{\no}{{\scriptscriptstyle\perp}}
\newcommand{\pd}{\partial}

\newcommand{\loc}{\text{{\rm loc}}}

\newcommand{\tS}{\widetilde S}
\newcommand{\E}{{\mathcal E}}

\def\barint_#1{\mathchoice
            {\mathop{\vrule width 6pt
height 3 pt depth -2.5pt
                    \kern -8.8pt
\intop}\nolimits_{#1}}%
            {\mathop{\vrule width 5pt height
3 pt depth -2.6pt
                    \kern -6.5pt
\intop}\nolimits_{#1}}%
            {\mathop{\vrule width 5pt height
3 pt depth -2.6pt
                    \kern -6pt
\intop}\nolimits_{#1}}%
            {\mathop{\vrule width 5pt height
3 pt depth -2.6pt
          \kern -6pt \intop}\nolimits_{#1}}}

\usepackage{color}

\definecolor{gr}{rgb}   {0.,   0.8,   0. } 
\definecolor{bl}{rgb}   {0.,   0.5,   1. } 
\definecolor{mg}{rgb}   {0.7,  0.,    0.7}

\begin{document}

\title[Cauchy non-integral formulas]{Cauchy non-integral formulas}
\author[Andreas Ros\'en]{Andreas Ros\'en$\,^1$}
\thanks{$^1\,$Formerly Andreas Axelsson}
\address{Andreas Ros\'en, Matematiska institutionen, Link\"opings universitet, 581 83 Link\"oping, Sweden}
\email{andreas.rosen@liu.se}

\begin{abstract}
We study certain generalized Cauchy integral formulas for gradients of solutions to second order divergence form elliptic systems, which appeared in recent work by P. Auscher and A. Ros\'en.
These are constructed through functional calculus and are in general beyond the scope of 
singular integrals.
More precisely, we establish such Cauchy formulas for solutions $u$ with gradient in
weighted $L_2(\R^{1+n}_+, t^{\alpha}dtdx)$ also in the case $|\alpha|<1$.
In the end point cases $\alpha= \pm 1$, we show how to apply 
Carleson duality results by T. Hyt\"onen and A. Ros\'en to establish such Cauchy formulas.
\end{abstract}



\thanks{Supported by Grant 621-2011-3744 from the Swedish research council, VR}
\maketitle

\section{Introduction}

A fundamental problem in modern harmonic analysis has been if the Cauchy singular integral on a Lipschitz curve defines an  $L_2$ bounded operator. Calder\'on~\cite{Ca} showed that this is indeed the case when the Lipschitz constant is small, and Coifman, McIntosh and Meyer~\cite{CMcM} showed boundedness for
any Lipschitz curve. In connection with the latter work it was also realized that boundedness
of the Cauchy integral, a problem in harmonic analysis, was equivalent to the Kato square root problem, a problem in operator theory, in one dimension.
The relation can be seen as follows. The Cauchy singular integral on the graph of a Lipschitz function $g:\R\to \R$ is given by
$$
 \text{p.v.} \frac i{\pi} \int_\R \frac{u(y) (1+ig'(y))dy}{(y+ig(y))-(x+ig(x))} = \sgn(BD)u(x),\qquad u\in L_2(\R).
$$
The operator theoretic expression $\sgn(BD)$ for the Cauchy integral on the right hand side is interpreted as follows. 
In $L_2(\R)$ we have the self-adjoint differential operator $D:= i\dd{}x$, or equivalently the Fourier multiplier $-\xi$, and the 
accretive multiplication operator $B= (1+ig'(x))^{-1}$. This yields a bisectorial operator $BD$ which was shown to have a bounded holomorphic functional calculus, see McIntosh and Qian~\cite{McQ}.
In particular the bounded symbol $\sgn(\lambda)= \pm 1$, $\pm \re\lambda>0$, yields an $L_2$-bounded Cauchy integral operator $\sgn(BD)$.
Note that when $B=I$, this formula is simply the Fourier relation $\tfrac i{\pi}\mF(\text{p.v.} 1/x)= \sgn(\xi)$.

The Kato square root estimate
$$
 \Big\| \sqrt{-\dd{}{x} a(x)\dd{}x} u\Big\|_2\approx \Big\|\dd ux\Big\|_2
$$ 
on the other hand follows from the boundedness of $\sgn(BD)$ for more 
general accretive coefficients $B$. In higher dimension the Kato square root estimate
$\|\sqrt{-\divv A \nabla}u\|_2\approx \|\nabla u\|_2$ follows from a 
similar estimate $\|\sgn(BD)\|<\infty$, with 
$B=\begin{bmatrix} I & 0 \\ 0 & A\end{bmatrix}$ and
$
  D= \begin{bmatrix} 0 & \divv \\ -\nabla & 0 \end{bmatrix}.
$
A major difficulty in higher dimension, $n\ge 2$, is that $D$
has an infinite dimensional null space. 
Note that $\begin{bmatrix} 0 & \dd{}x \\ -\dd{}x & 0 \end{bmatrix}= i \dd{}x$ when $n=1$, if we identify ranges $\R^2=\C$.
In higher dimension, the Kato square root problem on $\R^n$ was solved by 
Auscher, Hofmann, Lacey, McIntosh and Tchamitchian~\cite{AHLMcT} and the more general
result that operators of the form $BD$ have bounded holomorphic functional calculi was proved by Axelsson, Keith and McIntosh~\cite{AKMc}.

Coming back to the Cauchy integral, in this paper we study certain generalized Cauchy type operators
which appeared in recent work by Auscher and Axelsson~\cite{AA1}.
More precisely, the aim is on one hand to give some complementary results for these Cauchy operators on certain 
weighted $L_2$-space between the end point cases studied in \cite{AA1}, and on the other hand 
to show duality results in these end point cases, using results of Hyt\"onen and Ros\'en~\cite{HR}.

Our Cauchy operators are constructed in the above spirit, by applying suitable bounded and holomorphic symbols to an underlaying differential operator like $BD$.
We shall even need to apply more general operator valued symbols to the differential operator which,
changing the setup slightly, will be of the form $DB= B^{-1}(BD)B$.

To formulate the problem, consider a divergence form second order elliptic system
$$
   \divv_{t,x} A(t,x) \nabla_{t,x} u=0
$$
in the upper half space $\R^{1+n}_+:= \sett{(t,x)}{t>0,x\in \R^n}$, $n\ge 1$.
We assume that $u: \R^{1+n}_+\to \C^{m}$, $m\ge 1$, is vector valued and that coefficients
$A\in L_\infty(\R^{1+n}_+; \mL(\C^{(1+n)m}))$ are accretive in the sense that there exists $\kappa>0$ such that
$$
  \inf_{v\in \C^{(1+n)m}\setminus\{0\}}  \re(A(t,x)v,v)/|v|^2\ge \kappa
$$
for almost every $(t,x)\in\R^{1+n}_+$.
With minor modifications, all our results are valid under a weaker G\aa rding type inequality, uniformly in $t$. See \cite{AA1}.

A natural gradient of solutions $u$ is the {\em conormal gradient}
$$
  \nabla_A u:= \begin{bmatrix} \pd_{\nu_A}u \\ \nabla_\ta u \end{bmatrix},
$$
where $\pd_{\nu_A}u= (A\nabla_{t,x}u)_\no$ denotes the conormal derivative and $\nabla_\ta u= \nabla_x u$ denotes the tangential gradient of $u$. Similarly, $\divv_\ta= \divv_x$ and $\curl_\ta= \curl_x$ will denote tangential divergence and curl. We write $v_\no$ and $v_\ta$ for the parts of a vector $v$ normal and tangential to the boundary.

\begin{que}
   For solutions to a given divergence form equation $\divv_{t,x} A(t,x) \nabla_{t,x} u=0$ as above,
   is there a Cauchy type formula
 $$
    \nabla_A u|_{\R^n}\mapsto \nabla_A u|_{\R^{1+n}_+}
 $$
 for the conormal gradient?
\end{que}


To answer this question, we first need to specify function spaces for $\nabla_A u$.
We shall use the following natural subspaces of $L_2^\loc(\R^{1+n}_+)$.
Here and below, we often suppress the range of functions in notation, for example 
$L_2(\R^{1+n}_+)= L_2(\R^{1+n}_+; \C^{(1+n)m})$.
We also write $\|\cdot\|_{L_2(\R^n)}= \|\cdot\|_2$.

\begin{defn}
Define
$$
  N_{2,2}(\R^{1+n}_+):= \sett{f: \R^{1+n}_+\to \C^{(1+n)m}}{\|N(W_2 f)\|_2<\infty},
$$
using the non-tangential maximal function 
$N f(x):= \sup_{|y-x|<s} |f(s,y)|$
and $L_2$ Whitney averages
$W_2 f(t,x):= t^{-(1+n)/2} \|f\|_{L_2(W(t,x))}$
over Whitney regions $W(t,x):= \sett{(s,y)}{1/2<s/t<2, |y-x|<t}$.

For $-1\le \alpha\le 1$, let
$$
  L_2(\R^{1+n}_+, t^\alpha):=\bigsett{f: \R^{1+n}_+\to \C^{(1+n)m}}{\iint_{\R^{1+n}_+} |f(t,x)|^2 t^{\alpha} dtdx<\infty}.
$$
\end{defn}

It was shown in \cite[Lem. 5.3]{AA1} that
\begin{equation}   \label{eq:l2locestsforn}
  \sup_{t>0} \frac 1t\int_t^{2t} \|f_s\|_2^2 ds \lesssim \|N(W_2 f)\|_2^2 \lesssim \int_0^\infty \|f_s\|_2^2 \frac {ds}s.
\end{equation}
We think of $N_{2,2}(\R^{1+n}_+)$ as a substitute for $L_2(\R^{1+n}_+, t^\alpha)$ in the endpoint case
$\alpha=-1$, which allows for non-zero traces.

To state our results, we next introduce the operators that we use.
For more details, see \cite{AA1}.
With the second order divergence form operator $\divv_{t,x} A(t,x) \nabla_{t,x}$ comes a first order
self-adjoint differential operator
$
  D:= \begin{bmatrix} 0 & \divv_\ta \\ -\nabla_\ta & 0 \end{bmatrix}
$
acting tangentially, parallel to the boundary $\R^n$, and a pointwise transformed coefficient matrix
\begin{equation}   \label{eq:defnB}
  B= \begin{bmatrix} a^{-1} & -a^{-1}b \\ ca^{-1} & d-ca^{-1}b \end{bmatrix}
  \qquad\text{if}\qquad A= \begin{bmatrix} a & b \\ c & d \end{bmatrix}.
\end{equation}
How these operators appear are explained in Section~\ref{sec:ops}.
They act on $\C^{(1+n)m}$-valued functions, written as column vectors with normal parts first and tangential parts second.
We write $f_t(x)= f(t,x)$ for such functions in $\R^{1+n}_+$, and similarly for the coefficients 
$B_t(x)= B(t,x)$.
Write 
$$
  \E_t(x)=\E(t,x):=  I - B_0(x)^{-1}B(t,x),
$$
where $B_0(x)$ are some $t$-independent accretive coefficients $B_0\in L_\infty(\R^n; \mL(\C^{(1+n)m}))$.
Often $B_0(x)= B(0,x)$, but not always. 

Our fundamental operator is $DB_0$. Both as an operator 
in $L_2(\R^n)$ and in $L_2(\R^{1+n}_+, t^\alpha)$, acting in the $x$-variable for each fixed $t>0$,
it defines a closed and densely defined operator with spectrum contained in a bisector
$S_\omega= S_{\omega+}\cup(-S_{\omega+})$, where
$$
  S_{\omega+}:= \sett{\lambda\in\C}{|\arg\lambda|\le\omega}\cup\{0\}, \qquad \omega<\pi/2.
$$
In \cite{AKMc} it was proved that $DB_0$ has a bounded holomorphic functional calculus,
which gives estimates of operators $b(DB_0)$ formed by applying holomorphic functions $b:S_\mu\to \C$,
$\omega<\mu$, to the operator $DB_0$.
In particular, we shall need the operators
\begin{gather*}
  \Lambda:= |DB_0|, \qquad
  e^{-t\Lambda}, t>0, \qquad
  E_0^\pm := \chi_\pm(DB_0), \\
  S f_t:= \int_0^t \Lambda e^{-(t-s)\Lambda} E_0^+ f_s ds + \int_t^\infty \Lambda e^{-(s-t)\Lambda} E_0^- f_s ds.
\end{gather*}
For the first three operators, we view $DB_0$ as an operator in $L_2(\R^n)$ and apply the scalar holomorphic functions
$\lambda\mapsto |\lambda|:= \pm \lambda$, $\pm \re\lambda>0$, $\lambda\mapsto e^{-t|\lambda|}$
and $\lambda\mapsto \chi_\pm(\lambda):= 1$ if $\pm\re \lambda>0$ and $0$ elsewhere.
For the definition of $S$, we view $DB_0$ as an operator in $L_2(\R^{1+n}_+,t^\alpha)$
and apply the operator-valued holomorphic function $\lambda\mapsto F(\lambda)$, where
$$
  F(\lambda)f_t:=  \int_0^t \lambda e^{-(t-s)|\lambda|} \chi_+(\lambda) f_s ds - \int_t^\infty \lambda e^{-(s-t)|\lambda|} \chi_-(\lambda) f_s ds.
$$
When $|\alpha|<1$, $F$ is a bounded function, and hence $S= F(DB_0)$ is a bounded operator
on $L_2(\R^{1+n}_+,t^\alpha)$.
For the endpoint spaces $\alpha=\pm 1$, see Section~\ref{sec:ops}.
For further details of this operational calculus, see \cite[Sec. 6]{AA1}.

\begin{thm}   \label{thm:main}
Consider first the case $|\alpha|<1$ and assume that $\divv_{t,x} A(t,x) \nabla_{t,x}u=0$
with estimates $\|f\|_{L_2(\R^{1+n}_+,t^\alpha)}<\infty$ of the conormal gradient $f:= \nabla_A u$.
Then there exists a function $h^+\in E_0^+L_2(\R^n)$, such that
\begin{equation}   \label{eq:inteqn}
  f_t= \Lambda^\sigma e^{-t\Lambda}E_0^+ h^++  S\E_t f_t, \qquad \sigma:= (\alpha+1)/2.
\end{equation}
From this follows estimates
\begin{equation}   \label{eq:supest}
  \sup_{t>0} t^{-1} \int_t^{2t} \|\Lambda^{-\sigma}f_s\|_2^2 ds\lesssim \|f\|_{L_2(\R^{1+n}_+,t^\alpha)}
\end{equation}
and we have limits
\begin{equation}   \label{eq:dinitraces}
   \lim_{t\to 0+} t^{-1}\int_t^{2t}\|\Lambda^{-\sigma}f_s-h\|_2^2 ds=0= \lim_{t\to\infty} t^{-1}\int_t^{2t}\|\Lambda^{-\sigma}f_s\|_2^2 ds,
\end{equation}
where $h:= h^+ + \int_0^\infty \Lambda^{1-\sigma} e^{-s\Lambda} E_0^-\E_s f_s ds$, $\|h\|_2\lesssim \|f\|_{L_2(\R^{1+n}_+,t^\alpha)}$.
If furthermore $\alpha>0$, then we have the pointwise $L_2(\R^n)$ estimates $\sup_{t>0} \|\Lambda^{-\sigma}f_t\|_2\lesssim \|f\|_{L_2(\R^{1+n}_+,t^\alpha)}$
and limits
\begin{equation}   \label{eq:tracesl2}
\lim_{t\to 0+}\|\Lambda^{-\sigma}f_t-h\|_2=0= \lim_{t\to\infty} \|\Lambda^{-\sigma}f_t\|_2.
\end{equation}

Conversely, if $\|\E\|_{L_\infty(\R^{1+n}_+)}$ is sufficiently small, then the Cauchy type formula
\begin{equation}   \label{eq:cauchyform}
  f_t:= (I- S\E)^{-1} \Lambda^\sigma e^{-t\Lambda}E_0^+ h^+, \qquad h^+\in E_0^+L_2(\R^n),
\end{equation}
constructs a function $f$ with $\|f\|_{L_2(\R^{1+n}_+,t^\alpha)}\lesssim \|h^+\|_2$, which is the 
conormal gradient $f= \nabla_A u$ of a solution $u$ to $\divv_{t,x} A(t,x) \nabla_{t,x}u=0$.

When $\alpha=+1$, the above holds with the following changes.
We need to assume throughout that $\|\E\|_*<\infty$, and for the converse statement that $\|\E\|_*$
is sufficiently small, where $\|\cdot\|_*$ denotes the Carleson--Dahlberg norm from Definition~\ref{defn:carleson}.
Here $\sigma=1$ and we have traces in $L_2(\R^n)$ sense.

When $\alpha=-1$, the above holds with the following changes.
We need to assume throughout that $\|\E\|_*<\infty$, and for the converse statement that $\|\E\|_*$
is sufficiently small. Furthermore we need to replace $L_2(\R^{1+n}_+, t^\alpha)$ throughout by
$N_{2,2}(\R^{1+n}_+)$.
Here $\sigma=0$ and we have traces only in the square Dini sense \eqref{eq:dinitraces}.
\end{thm}

Note that the trace spaces for $f$ are exactly the fractional homogeneous Sobolev spaces $\dot H^{-\sigma}(\R^n)$. We record the following result, proved in Section~\ref{sec:proof}.

\begin{prop}   \label{prop:sob}
  Let $0\le \sigma\le 1$.
  For all $f\in \dom(\Lambda^{-\sigma})$ with $\curl_\ta f_\ta=0$, we have
$$
  \|\Lambda^{-\sigma} f\|_2\approx \|f\|_{\dot H^{-\sigma}}.
$$
\end{prop}

Note that in the cases $\alpha=\pm 1$, the $t$-independent coefficients $B_0$ are uniquely determined by
$B$, see \cite[Lem. 2.2]{AA1}. When $|\alpha|<1$, this is not the case.

We also remark that the representation formula \eqref{eq:inteqn} can be used to prove various 
other estimates of solutions. See \cite{AA1}.

\begin{ex}
To recognize \eqref{eq:cauchyform} as a Cauchy formula, consider the special case $A=B=I$,
$n=1=m$ and $\alpha=-1$.
Then $\nabla u$ will be anti-analytic, and hence
$$
   \nabla u(t,x)= \frac i{2\pi}\int_\R \frac {\nabla u(0,y)}{y-x+it} dy
$$
is the Cauchy type reproducing formula we look for.
Letting $\E=0$, $\sigma=0$ and $B=I$ 
in Theorem~\ref{thm:main}, formula \eqref{eq:cauchyform} reduces to
$$
  \nabla u= e^{-t|D|}\chi_+(D) h^+.
$$
To compare these two expressions, note the Fourier relation $\tfrac i{2\pi}\mF_x((it-x)^{-1})= e^{-t|\lambda|}\chi_+(-\lambda)$,
and that $D$ is the Fourier multiplier $-\xi$.

As less trivial example, we consider the special case $m=1$, $A$ being real $t$-independent coefficients
and $\alpha=-1$. Then it was shown in \cite{R1} that 
$$
  e^{-t\Lambda} E_0^+ h^+(x) = \nabla_A \int_{\R^
  n} \Gamma_{(0,y)}(t,x) h^+(y) dy
$$
is the conormal gradient of the single layer potential, for normal vector / scalar fields $h\in L_2(\R^n)$.
Here $\Gamma_{(s,y)}$ denotes the fundamental solution to $\divv_{t,x} A(x)\nabla_{t,x}$ in $\R^{1+n}$ with pole at $(s,y)$.
Note that in the case of the Laplace equation $A=I$, 
$$
\nabla_{t,x} \Gamma_{(0,0)}(t,x)= \frac 1{\sigma_n}\frac{(t,x)}{(t^2+|x|^2)^{(1+n)/2}}
$$
is the Cauchy/Riesz kernel, $\sigma_n$ denoting the area of the unit sphere in $\R^{1+n}$.

Finally, we remark that for general systems, $m\ge 2$, and general coefficients $A$, the operators 
defined from $DB_0$ by functional calculus are usually beyond the scope of singular integrals. 
For example, the known constructions and estimates of the fundamental solution $\Gamma_{(s,y)}$
require De Giorgi-Nash local H\"older estimates of solutions to the divergence form equation,
which may fail for systems, $m\ge 2$.
\end{ex}

The end point cases $\alpha=\pm 1$ and the estimate of $S$ in the case $|\alpha|<1$ was proved in \cite{AA1}. In this paper we supply the details of the remaining results stated for $|\alpha|<1$ in Section~\ref{sec:proof} and a simplified proof of the estimate for $S$ in the case $\alpha=+1$ in Section~\ref{sec:ops}.
In the final Section~\ref{sec:bvp}, we make some remarks on applications to the Neumann and Dirichlet 
problem for divergence form equations.

\section{Carleson estimates of operators}  \label{sec:ops}

The aim with this section is to give a simplified proof of the estimates of the singular integral
operator $S=S_A$ from \cite{AA1} in the case $\alpha=+1$, using Carleson duality results from \cite{HR}.
We start by deriving the integral equation \eqref{eq:inteqn} for the conormal gradient $f=\nabla_A u$ from the divergence form second order differential equation $\divv_{t,x} A(t,x) \nabla_{t,x}u=0$ for the potential $u$.

Splitting $A$ as in \eqref{eq:defnB}, we have $f_\no= a\pd_t u+ b\nabla_\ta u$ and $f_\ta= \nabla_\ta u$.
Thus the divergence form equation, in terms of $f$ reads
$$
  \pd_t f_\no + \divv_\ta ( c a^{-1}(f_\no- b f_\ta) + d f_\ta )= 0.
$$
The condition that $f$ is the conormal gradient of a function $u$, determined up to constants, we express 
as the curl-free condition
$$
\begin{cases}
  \pd_t f_\ta= \nabla_\ta( a^{-1}(f_\no- b f_\ta) ), \\
  \curl_\ta f_\ta =0.
\end{cases}
$$
In vector notation, we have 
$$
  \pd_t  \begin{bmatrix}f_\no \\ f_\ta \end{bmatrix}+
  \begin{bmatrix} 0 & \divv_\ta \\ -\nabla_\ta & 0  \end{bmatrix}
   \begin{bmatrix} a^{-1} & -a^{-1}b \\ ca^{-1} & d- ca^{-1}b  \end{bmatrix} \begin{bmatrix}f_\no \\ f_\ta \end{bmatrix}=0,
$$
together with the constraint $\curl_\ta f_\ta=0$, or in short hand notation
$$
   \pd_t f_t + DB_t f_t=0, \qquad f_t \in \clos{\ran(D)}=: \mH.
$$
With the $t$-independent coefficients $B_0$, we rewrite the equation as 
\begin{equation}    \label{eq:odepert}
  \pd_t f_t + DB_0 f_t = DB_0 \E_t f_t.
\end{equation}
We shall use freely known properties of operators of the form $DB_0$. See \cite{AAM, AA1}.
In particular $DB_0$ is a (non-injective if $n\ge 2$) bisectorial operator and
$L_2(\R^n)= \nul(DB_0)\oplus \mH$, where 
$$
  \mH= E_0^+L_2 \oplus E_0^- L_2.
$$
We now integrate the vector-valued ordinary differential equation \eqref{eq:odepert} for $f_t\in \mH$. Applying the projections $E_0^\pm$, we have
$$
  \begin{cases}
    \pd_t f^+_t + \Lambda f^+_t = \Lambda E_0^+ \E_t f_t, \\
    \pd_t f^-_t - \Lambda f^-_t = -\Lambda E^-_0 \E_t f_t,
  \end{cases}
$$ 
where $f_t^\pm := E_0^\pm f_t$.

Formally, assuming $\lim_{t\to 0^+} f_t = f_0$ and $\lim_{t\to\infty} f_t =0$, we integrate these two equations $$
  \begin{cases}
    f_0^+- e^{-t\Lambda} f_0^+ = \int_0^t \Lambda e^{-(t-s)\Lambda} E_0^+\E_s f_s ds, \\
    0- f_t^-= -\int_t^\infty \Lambda e^{-(s-t)\Lambda} E_0^- \E_s f_s ds,
  \end{cases}
$$ 
and subtraction yields the integral equation
$$
  f_t= e^{-t\Lambda} E_0^+ f_0 + S \E_t f_t.
$$
In Section~\ref{sec:proof} we show by a rigorous argument that, depending on the function space 
for $f$, integration indeed yields this equation with $E_0^+f_0= \Lambda^\sigma h^+$, for some 
$h^+ \in E_0^+L_2(\R^n)$.
In this section, we discuss estimates of the singular integral operator $S$ and the multiplier $\E$, 
in particular in the case $\alpha=+1$. 

\begin{defn}    \label{defn:carleson}
  For functions $f$ in $\R^{1+n}_+$, define the Carleson functional 
  $$
  Cf(x):= \sup_{r>0}r^{-n} \iint_{|y-x|<r-s}|g(s,y)|dsdy
  $$
  and the area functional 
  $
  Af(x):= \iint_{|y-x|<s} |f(s,y)|s^{-n}dsdy$, $x\in \R^n$.
Define the Banach space
$$
  C_{2,2}(\R^{1+n}_+):= \sett{f: \R^{1+n}_+\to \C^{(1+n)m}}{\|C(W_2 f)\|_{L_2(\R^n)}\approx \|A(W_2 f)\|_{L_2(\R^n)}<\infty}.
$$
\end{defn}

The equivalence of the Carleson and area functionals
$$
  \|Cg\|_{L_p(\R^n)}\approx \|Ag\|_{L_p(\R^n)}, \qquad 1<p<\infty,
$$
follows from  \cite[Thm. 3]{CMS}.
From \cite[Thm. 3.2, 3.1]{HR}, we recall the following duality result.

\begin{prop}   \label{prop:duality}
  The Banach space $N_{2,2}(\R^{1+n}_+)$ is the dual space of $C_{2,2}(\R^{1+n}_+)$
   under the $L_2(\R^{1+n}_+)$ pairing.
  The space $C_{2,2}(\R^{1+n}_+)$ is not reflexive, that is $N_{2,2}(\R^{1+n}_+)^*\supsetneqq C_{2,2}(\R^{1+n}_+)$. 
\end{prop}

By Proposition~\ref{prop:duality} and \eqref{eq:l2locestsforn} we have  continuous inclusions
$$
  L_2(\R^{1+n}_+,t^{-1})\subset N_{2,2}(\R^{1+n}_+)\qquad \text{and}\qquad
  C_{2,2}(\R^{1+n}_+)\subset  L_2(\R^{1+n}_+,t).
$$

The estimates for the singular integral $S$ are as follows. 

\begin{thm}   \label{thm:Sest}
If $|\alpha|<1$, then 
$
  S: L_2(\R^{1+n}_+, t^\alpha)\to L_2(\R^{1+n}_+,t^\alpha)
$
is a bounded operator. 
For $\alpha=\pm 1$, we have bounded operators
$
  S: L_2(\R^{1+n}_+, t^{-1})\to N_{2,2}(\R^{1+n}_+)
$
and
$
  S:  C_{2,2}(\R^{1+n}_+)\to  L_2(\R^{1+n}_+,t).
$
\end{thm}

This result was proved in \cite{AA1}, with the exception that the Carleson space $C_{2,2}(\R^{1+n}_+)$
was not known there and in the endpoint case $\alpha=+1$ only the estimate
$$
  \|S\E\|_{L_2(\R^{1+n}_+, t)\to L_2(\R^{1+n}_+, t)}\lesssim \|\E\|_*
$$
was proved. We here survey the proof from \cite{AA1} and supply the missing estimates of 
$S$ and $\E$ separately in the case $\alpha=1$.

\begin{proof}
  First recall the rigorous definition of the singular integral $S$ from \\ \cite[Sec. 6,7]{AA1}.
  For fixed $\epsilon>0$, define truncated singular integral operators
$$
  S_\epsilon f_t
  := \int_0^t \eta_\epsilon^+(t,s) \Lambda e^{-(t-s)\Lambda} E_0^+ f_s ds + \int_t^\infty \eta^-_\epsilon(t,s) \Lambda e^{-(s-t)\Lambda} E_0^- f_s ds,
$$
where $\eta_\epsilon^\pm$ are compactly supported approximations of the characteristic functions 
of the triangles $\sett{(t,s)}{0<s<t}$ and $\sett{(t,s)}{0<t<s}$.
Then $S_\epsilon: L_1^\loc(\R_+; L_2(\R^n))\to L_\infty^c(\R_+; L_2(\R^n))$ is a well defined operator.
More precisely, we define $\eta^0(t)$ to be the piecewise linear continuous function with support
$[1,\infty)$, which equals $1$ on $(2,\infty)$ and is linear on $(1,2)$.
Then let $\eta_\epsilon(t):= \eta^0(t/\epsilon)(1- \eta^0(2\epsilon t))$ and
$\eta_\epsilon^\pm(t,s):= \eta^0(\pm (t-s)/\epsilon) \eta_\epsilon(t)\eta_\epsilon(s)$.

For $|\alpha|<1$, it was proved in \cite[Thm. 6.5]{AA1} that $S_\epsilon$ are uniformly bounded
and converge strongly to an operator $S$ in $L_2(\R^{1+n}_+,t^\alpha)$.
The idea of proof was to view $S_\epsilon$ as being constructed from the underlaying operator $DB_0$ in 
$L_2(\R^{1+n}_+,t^\alpha)$ by applying the operator-valued symbol $\lambda\mapsto F_\epsilon(\lambda)$, where
$$
  F_\epsilon(\lambda)f_t  := \int_0^t \eta_\epsilon^+(t,s) |\lambda| e^{-(t-s)|\lambda|} \chi_+(\lambda) f_s ds + \int_t^\infty \eta^-_\epsilon(t,s) |\lambda| e^{-(s-t)|\lambda|} \chi_-(\lambda) f_s ds,
$$
yielding $S_\epsilon= F_\epsilon(DB_0)$.
It was shown by Schur estimates that 
$$
\sup_{\epsilon>0, \lambda\in S_{\mu}}\|F_\epsilon(\lambda)\|_{L_2(\R^{1+n}_+,t^\alpha)\to L_2(\R^{1+n}_+,t^\alpha)}<\infty,
$$
for any $\omega<\mu<\pi/2$. 
Moreover, for any fixed $\epsilon>0$ there is decay
$$
\lim_{S_\mu\ni\lambda\to 0, \infty}\|F_\epsilon(\lambda)\|_{L_2(\R^{1+n}_+,t^\alpha)\to L_2(\R^{1+n}_+,t^\alpha)}=0,
$$
and $F_\epsilon(\lambda)f\to F(\lambda)f$ for each $\lambda\in S_\mu$ and $f\in L_2(\R^{1+n}_+,t^\alpha)$ as $\epsilon \to 0$.
As shown in \cite[Sec. 6.1]{AA1}, from this and square function estimates for $DB_0$ it follows that 
$S_\epsilon= F_\epsilon(DB_0)$ are uniformly bounded operators in $L_2(\R^{1+n}_+,t^\alpha)$ which 
converge strongly to $S= F(DB_0)$.

At the end point space $\alpha=-1$, 
the above bounds fail on $L_2(\R^{1+n}_+, t^{-1})$ and we split the operator as 
$$
  S_\epsilon= F^1_\epsilon(DB_0) + F^2_\epsilon(DB_0),
$$
where 
\begin{multline*}
  F^1_\epsilon(\lambda)f_t  := \int_0^t \eta_\epsilon^+(t,s) |\lambda| e^{-(t-s)|\lambda|} \chi_+(\lambda) f_s ds \\
  + \int_t^\infty \eta^-_\epsilon(t,s) |\lambda| (e^{-(s-t)|\lambda|}- e^{-(s+t)|\lambda|}) \chi_-(\lambda) f_s ds \\
  -\int_0^{t+2\epsilon}(\eta_{\epsilon}(t)\eta_\epsilon(s)-\eta^-_\epsilon(t,s))|\lambda| e^{-(s+t)|\lambda|} \chi_-(\lambda) f_s ds
\end{multline*}
and 
$$
  F_\epsilon^2(\lambda)f_t:= \eta_\epsilon(t) e^{-t|\lambda|} \int_0^\infty \eta_\epsilon(s) |\lambda| e^{-s|\lambda|} \chi_-(\lambda) f_s ds.
$$
On one hand, the term $F^1_\epsilon$ can be treated as in the case $|\alpha|<1$, yielding a bounded operator
$F^1(DB_0): L_2(\R^{1+n}_+, t^{-1})\to L_2(\R^{1+n}_+, t^{-1})$.
On the other hand, the term $F^2_\epsilon$ factorizes as a bounded operator
$$
  L_2(\R^{1+n}_+, t^{-1}) \to L_2(\R^n)\to N_{2,2}(\R^{1+n}_+),
$$
where the second factor does not converge strongly in $\mL(L_2(\R^n),  N_{2,2}(\R^{1+n}_+))$ but only 
in $\mL(L_2(\R^n), L_2(a,b; L_2(\R^n))$ for fixed but arbitrary $0<a<b<\infty$.

Now finally consider the end point space $\alpha=+1$. Here the appropriate splitting is 
$$
  S_\epsilon= F^3_\epsilon(DB_0) + F^4_\epsilon(DB_0),
$$
where 
\begin{multline*}
  F^3_\epsilon(\lambda)f_t  := 
  \int_0^t \eta^+_\epsilon(t,s) |\lambda| (e^{-(t-s)|\lambda|}- e^{-(t+s)|\lambda|}) \chi_+(\lambda) f_s ds \\
  -\int_{t-2\epsilon}^\infty (\eta_{\epsilon}(t)\eta_\epsilon(s)-\eta^+_\epsilon(t,s))|\lambda| e^{-(t+s)|\lambda|} \chi_+(\lambda) f_s ds \\
  + \int_t^\infty \eta^-_\epsilon(t,s) |\lambda| e^{-(s-t)|\lambda|} \chi_-(\lambda) f_s ds
\end{multline*}
and 
$$
  F_\epsilon^4(\lambda)f_t:= \eta_\epsilon(t) |\lambda| e^{-t|\lambda|} \int_0^\infty \eta_\epsilon(s) e^{-s|\lambda|} \chi_+(\lambda) f_s ds.
$$
(Note the duality $F^3_\epsilon= (F^1_\epsilon)^*$ and $F^4_\epsilon= (F^2_\epsilon)^*$.)
On one hand, the term $F^3_\epsilon$ can be treated as in the case $|\alpha|<1$, yielding a bounded operator
$F^3(DB_0): L_2(\R^{1+n}_+, t)\to L_2(\R^{1+n}_+, t)$.
On the other hand, the term $F^4_\epsilon$ factorizes as a bounded operator
$$
  C_{2,2}(\R^{1+n}_+)\to L_2(\R^n)\to L_2(\R^{1+n}_+, t).
$$
The bounds of the second factor follow directly from square function estimates, whereas
the bounds of the first factor follow by duality from the non-tangential maximal estimates used in the
case $\alpha=-1$. We have
\begin{multline*}
  \left(\phi,\int_0^\infty \eta_\epsilon(s)  e^{-s\Lambda} E_0^+ f_s ds \right)
  = \int_0^\infty(  e^{-s\Lambda^*} ((E_0^+)^* \phi),  f_s)\eta_\epsilon(s) ds \\
  \lesssim  \|e^{-s\Lambda^*} ((E_0^+)^* \phi)\|_{N_{2,2}(\R^{1+n}_+)} \|f\|_{C_{2,2}(\R^{1+n}_+)}\lesssim 
  \|\phi\|_2 \|f\|_{C_{2,2}(\R^{1+n}_+)}.
\end{multline*}
Note that in this case, both factors converge strongly as $\epsilon\to 0$, using that compactly supported functions in $\R^{1+n}_+$ are dense in $C_{2,2}(\R^{1+n}_+)$. Proof of this density result is in \cite[Lem. 2.5]{HR}.
Thus $S_\epsilon\to S$ strongly as operators $C_{2,2}(\R^{1+n}_+)\to L_2(\R^{1+n}_+, t)$.
\end{proof}

For a multiplier $\E$, clearly
$$
  \sup_{\|f\|_{L_2(\R^{1+n}_+,t^\alpha)}=1} \|\E f\|_{L_2(\R^{1+n}_+,t^\alpha)}
  = \|\E\|_{L_\infty(\R^{1+n}_+)}
$$
for any $\alpha$. For $\alpha=\pm 1$ we have the following more refined Carleson multiplier estimates.
Define the Carleson--Dahlberg norm 
$$
    \|\E\|_*:=\| C W_\infty(\tfrac {\E^2}t) \|_\infty^{1/2},
$$
using $L_\infty$-Whitney averages $W_\infty g(t,x):=\|g\|_{L_\infty(W(t,x))}$.

\begin{thm}
The following are equivalent.
\begin{itemize}
\item[{\rm (i)}] $\E: f(t,x)\to \E(t,x)f(t,x)$ is bounded $N_{2,2}(\R^{1+n}_+)\to L_2(\R^{1+n}_+,t^{-1})$.
\item[{\rm (ii)}] $\E: f(t,x)\to \E(t,x)f(t,x)$ is bounded $L_2(\R^{1+n}_+,t)\to C_{2,2}(\R^{1+n}_+)$.
\item[{\rm (iii)}] $\E$ has the Carleson--Dahlberg estimate $\|\E\|_*<\infty$.
\end{itemize}
If this hold, then 
$$
\|\E\|_*\approx \sup_{\|f\|_{N_{2,2}(\R^{1+n}_+)}=1}\|\E f\|_{L_2(\R^{1+n}_+,t^{-1})}\approx \sup_{\|f\|_{L_2(\R^{1+n}_+,t)}=1}\|\E f\|_{C_{2,2}(\R^{1+n}_+)}.
$$
\end{thm}

\begin{proof}
The equivalence of (i) and (iii) follows from \cite[Thm. 3.1]{HR}.
The equivalence of (i) and (ii) follows from Proposition~\ref{prop:duality} since
$\E^*: N_{2,2}(\R^{1+n}_+)\to L_2(\R^{1+n}_+,t^{-1})$ is the adjoint of
$\E: L_2(\R^{1+n}_+,t)\to C_{2,2}(\R^{1+n}_+)$ under the $L_2(\R^{1+n}_+)$ pairing.
\end{proof}

\section{Proof of the Cauchy formula}   \label{sec:proof}

In this section, we prove Theorem~\ref{thm:main}. The end point cases $\alpha=\pm 1$ were proved 
in \cite{AA1}. Thus it remains to show the results for $|\alpha|<1$. We remark though that 
in \cite{AA1}, in the case $\alpha=1$ it was only shown that $S\E: L_2(\R^{1+n}_+, t)\to L_2(\R^{1+n}_+, t)$.
With the intermediate Carleson space $C_{2,2}(\R^{1+n}_+)$ available now from \cite{HR},
we have the refined mapping result
$$
  \E:  L_2(\R^{1+n}_+, t)\to C_{2,2}(\R^{1+n}_+)\qquad \text{and}\qquad
  S: C_{2,2}(\R^{1+n}_+)\to L_2(\R^{1+n}_+, t),
$$
if $\|\E\|_*<\infty$, proved in Section~\ref{sec:ops}.

\begin{proof}[Proof of the representation formula \eqref{eq:inteqn}]
Assume that $|\alpha|<1$ and that $\divv_{t,x} A(t,x) \nabla_{t,x}u=0$
with estimates $\|f\|_{L_2(\R^{1+n}_+,t^\alpha)}<\infty$ of the conormal gradient $f:= \nabla_A u$.
As in \cite[proof of Thm. 8.2]{AA1}, for any $\epsilon>0$, integration of \eqref{eq:odepert} gives
\begin{multline*}
   S_\epsilon\E f_t=  \epsilon^{-1}\int_\epsilon^{2\epsilon} e^{-s\Lambda} (E_0^+ f_{t-s}+ E_0^- f_{t+s}) ds
  -e^{-t\Lambda} \epsilon^{-1} \int_\epsilon^{2\epsilon}  E_0^+ f_s ds \\ +
  \epsilon^{-1} \int_\epsilon^{2\epsilon} (e^{-t\Lambda} -e^{-(t-s)\Lambda} )E_0^+ f_s ds
  -2\epsilon \int_{(2\epsilon)^{-1}}^{\epsilon^{-1}} e^{-(s-t)\Lambda} E_0^- f_s ds=: I-II+III-IV,
\end{multline*}
with equality in $L_2(a,b; \mH)$ for any fixed $0<a<b<\infty$,
where $I$ converges to $f$ in $L_2(a,b; \mH)$.
Using estimates $\|e^{-t\Lambda} -e^{-(t-s)\Lambda}\|\lesssim s$ for III and $\|e^{-(s-t)\Lambda}\|\lesssim 1$
for IV, we obtain
$$
  \|III\|_2 \lesssim \epsilon^{(1-\alpha)/2}\to 0\qquad\text{and}\qquad \|IV\|_2\lesssim \epsilon^{(1+\alpha)/2}\to 0,
$$
as $\epsilon\to 0$, for each fixed $a\le t\le b$.
Since $S_\epsilon \E f$ converges in $L_2(\R^{1+n}_+, t^\alpha)$ by \cite[Thm. 6.5]{AA1}, it follows from the
equation that II converges in $L_2(a,b; \mH)$ for any fixed $0<a<b<\infty$.

Write $\tilde f_t:= \lim_{\epsilon\to 0} e^{-t\Lambda} \epsilon^{-1} \int_\epsilon^{2\epsilon}  E_0^+ f_s ds$.
Since $\tilde f= f-S\E f$ by the equation, $\|\tilde f\|_{L_2(\R^{1+n}_+, t^\alpha)}\lesssim \|f\|_{L_2(\R^{1+n}_+, t^\alpha)}<\infty$.
Using the identity $I = 4\int_0^\infty (s\Lambda e^{-s\Lambda})^2 ds/s$ on $\mH$, and $e^{-t\Lambda}\tilde f_s= \tilde f_{t+s}$ so that $\|\tilde f_{t+s}\|_2\lesssim \|\tilde f_s\|_2$, we estimate
\begin{multline*}
   ((\Lambda^*)^{-\sigma}\phi, \tilde f_t)= 4\int_0^\infty ((s\Lambda^*)^{2-\sigma} e^{-s\Lambda^*}\phi, s^{\sigma} e^{-s\Lambda} \tilde f_t) \frac{ds}s \lesssim \\
   \|\phi\|_2 \left( \int_0^\infty \|\tilde f_{t+s}\|^2 s^\alpha ds\right)^{1/2} \lesssim \|\phi\|_2 \|f\|_{L_2(\R^{1+n}_+, t^\alpha)},
\end{multline*}
for any $\phi\in\dom((\Lambda^*)^{-\sigma})$.
Thus $\tilde f_t= \Lambda^\sigma h_t$ with $\sup_{t>0}\|h_t\|_2\lesssim \|f\|_{L_2(\R^{1+n}_+, t^\alpha)}$.
We now want to let $s\to 0$ in the identity
$$
  \tilde f_{t+s}= e^{-t\Lambda}\tilde f_s= \Lambda^\sigma e^{-t\Lambda} h_s.
$$
Integrating against a test function $\phi\in L_2(a,b; \mH)$, we get
$$
  \int_a^b (\phi_t, \tilde f_{t+s}) dt= \left( \int_a^b \Lambda^* e^{-t\Lambda^*}\phi_t dt, h_s \right).
$$
By continuity of translations in $L_2(a,b;\mH)$, the left hand side converges. 
Since functions of the form $\int_a^b \Lambda^* e^{-t\Lambda^*}\phi_t dt$ are dense in $L_2(\R^n)$ and
$h_s$ are uniformly bounded, $h_s\to h$ weakly in $L_2(\R^n)$ when $s\to 0$.
We conclude that 
$f_t-S\E f_t= \tilde f_t= \Lambda^\sigma e^{-t\Lambda} h$ in $L_2(\R^{1+n}_+,t^\alpha)$.
\end{proof}

\begin{proof}[Proof of the estimate \eqref{eq:supest}]
 By equation  \eqref{eq:inteqn} it suffices to estimate the weakly singular integral operator
 \begin{multline*}
   \tS f_t := \int_0^t \Lambda^{1-\sigma} e^{-(t-s)\Lambda} E_0^+ f_s ds + \int_t^\infty \Lambda^{1-\sigma} e^{-(s-t)\Lambda} E_0^- f_s ds \\
   = \int_{t/2<s<2t} \Lambda^{1-\sigma} e^{-|s-t|\Lambda} E_0^{\sgn(t-s)}f_s ds
   + \int_0^{t/2}\Lambda^{1-\sigma} e^{-(t-s)\Lambda}(I-e^{-2s\Lambda}) E_0^+ f_s ds  \\
   + \int_{2t}^{\infty}\Lambda^{1-\sigma} e^{-(s-t)\Lambda}(I-e^{-2t\Lambda}) E_0^- f_s ds  
   + e^{-t\Lambda}\int_{\R\setminus [t/2,2t]} \Lambda^{1-\sigma} e^{-s\Lambda} E_0^{\sgn(t-s)} f_s ds \\
   =: I+II+III+IV.
\end{multline*}
We first estimate IV by duality.
For $\phi\in L_2(\R^n)$, we have
$$
  |(IV, \phi)|\lesssim \int_0^\infty \| (s\Lambda)^{1-\sigma} e^{-s\Lambda^*} \phi \|_2 \|s^\sigma f_s\|_2 \frac {ds}s\lesssim \|\phi\|_2 \|f\|_{L_2(\R^{1+n}_+, t^\alpha)},
$$
so that $\sup_{t>0}\|IV\|_2\lesssim \|f\|_{L_2(\R^{1+n}_+, t^\alpha)}$.
To estimate II, we note that 
\begin{multline*}
  \|\Lambda^{1-\sigma} e^{-(t-s)\Lambda}(I-e^{-2s\Lambda})\| \\
  =  \| (s/(t-s)^{2-\sigma}) ((t-s) \Lambda)^{2-\sigma} e^{-(t-s)\Lambda}(I-e^{-2s\Lambda})/(s\Lambda) \| 
  \lesssim s/t^{2-\sigma}.
\end{multline*}
This gives
\begin{multline*}
  \sup_{t>0}\|II\|_2\lesssim \int_0^{t/2} s/t^{2-\sigma}\|f_s\|_2 ds  \\
  \lesssim \left( \int_0^{t/2} s^{2-\alpha} t^{2\sigma-4} ds\right)^{1/2} \|f\|_{L_2(\R^{1+n}_+, t^\alpha)}\lesssim  \|f\|_{L_2(\R^{1+n}_+, t^\alpha)}.
\end{multline*}
A similar estimate applies to III. We are left with the local weakly singular integral I, which we estimate
\begin{multline*}
  \|I\|_2\lesssim \int_{t/2<s<2t} \frac {\|f_s\|_2 ds}{|t-s|^{1-\sigma}} \\
  \lesssim \left( t^{-\alpha} \int_{t/2}^{2t} \frac{ds}{|t-s|^{2-2\sigma}} \right)^{1/2}\|f\|_{L_2(\R^{1+n}_+, t^\alpha)}\lesssim \|f\|_{L_2(\R^{1+n}_+, t^\alpha)}
\end{multline*}
if $\alpha>0$. If $\alpha\le 0$, we at least obtain the weaker estimate
\begin{multline*}
  t^{-1}\int_{t}^{2t} \left(  \int_{u/2<s<2u} \frac {\|f_s\|_2 ds}{|u-s|^{1-\sigma}} \right)^2 du \\
  \lesssim 
   t^{-1}\int_{t}^{2t} \left(  \int_{u/2<s<2u} \frac { ds}{|u-s|^{1-\sigma}} \right)
   \left(  \int_{u/2<s<2u} \frac {\|f_s\|_2^2 ds}{|u-s|^{1-\sigma}} \right) du \\
\lesssim 
   t^{-1}\int_{t}^{2t} u^\sigma
   \left(  \int_{u/2<s<2u} \frac {\|f_s\|_2^2 ds}{|u-s|^{1-\sigma}} \right) du \\
   \lesssim 
   t^{-1}\int_{t/2}^{4t} 
   \left(  \int_{s/2<u<2s} \frac {u^{\sigma} du}{|u-s|^{1-\sigma}} \right)\|f_s\|_2^2 ds
   \lesssim \|f\|_{L_2(\R^{1+n}_+, t^\alpha)}.
\end{multline*}
This proves the estimate of $\|\tilde S f_t\|_2$ and therefore of \eqref{eq:supest}.
\end{proof}

\begin{proof}[Proof of Theorem~\ref{thm:main}]
We supply the remaining arguments in the case $|\alpha|<1$.
Having established the representation formula \eqref{eq:inteqn}, write this as
$$
  f_t = \Lambda^\sigma(e^{-t\Lambda} E_0^+ + \tS \E_t f_t).
$$
The estimates above for $\tS$ give the stated estimates for $\|\Lambda^{-\sigma} f_t\|_2$.
For the traces, it remains to prove that 
$$
  \lim_{t\to 0} (\tS\E f)_t= \int_0^\infty \Lambda^{1-\sigma} e^{-s\Lambda} E_0^-\E_s f_s ds=: h^-, \qquad
  \lim_{t\to \infty} (\tS\E f)_t= 0,
$$
either in $L_2(\R^n)$ or square Dini sense.
By the established uniform bounds, we may assume that $f_t\ne 0$ only if $a\le t\le b$ for some $0<a<b<\infty$.
In this case,
 $$
     \tS\E_t f_t = \int_{a<s<\min(t,b)} \Lambda^{1-\sigma}e^{-(t-s)\Lambda} E_0^+ \E_s f_s ds +
      \int_{\max(t,a)<s<b} \Lambda^{1-\sigma}e^{-(s-t)\Lambda} E_0^- \E_s f_s ds.
 $$
Since  
$$
\int_a^b \|\Lambda^{1-\sigma}e^{-(t-s)\Lambda}E_0^+ \E_s f_s\|_2 ds
\lesssim \int_a^b (t-s)^{\sigma-1}\|f_s\|_2 ds\to 0
$$ 
when $t\to \infty$ and 
$$
\int_a^b \|\Lambda^{1-\sigma}(e^{-(s-t)\Lambda}-e^{-s\Lambda})E_0^- \E_s f_s\|_2 ds
\lesssim \int_a^b t \|f_s\|_2 ds\to 0
$$
as $t\to 0$, we have proved the traces.

The converse result is obtained by reversing the argument leading to the 
 representation formula \eqref{eq:inteqn}, outlined in Section~\ref{sec:ops}.
 Note that square function estimates give
$$
  \int_0^\infty \|\Lambda^\sigma e^{-t\Lambda} E_0^+ h^+\|_2^2 t^\alpha dt
  = \int_0^\infty \|(t\Lambda)^\sigma e^{-t\Lambda} E_0^+ h^+\|_2^2 \frac{dt}t\lesssim \|h^+\|_2^2,
$$
and if $\|S\E\|_{L_2(\R^{1+n}_+, t^\alpha)\to L_2(\R^{1+n}_+, t^\alpha)}\lesssim \|\E\|_{L_\infty(\R^{1+n}_+)}<1$, then $I-S\E$ is invertible on $L_2(\R^{1+n}_+, t^\alpha)$.
\end{proof}

\begin{proof}[Proof of Proposition~\ref{prop:sob}]
Consider the operator 
$\Lambda^*= |B_0^* D|$.
From square function estimates and accretivity of $B_0$ it follows that
$$
 \| |B_0^*D| g \|_2 \approx \| B_0^* D g\|_2 \approx \|Dg\|_2 \approx \||D|g\|_2,
$$
where 
$$
  |D|= \begin{bmatrix} (-\Delta_\ta)^{1/2} & 0 \\ 0 & (-\nabla_\ta \divv_\ta)^{1/2} \end{bmatrix}.
$$
Thus $\dom(\Lambda^*)= \dom(|D|)$, and interpolation shows that $\dom((\Lambda^*)^\sigma)= \dom(|D|^\sigma)$ with equivalence of norms.
See remark following \cite[Thm. 4.2]{AMcNhol}, where interpolation between homogeneous norms
gives the estimate $\|(\Lambda^*)^\sigma g\|_2\approx \||D|^\sigma g\|_2$,
for all $g\in \dom((\Lambda^*)^\sigma)= \dom(|D|^\sigma)$. 

By duality, we have for $f\in \dom(\Lambda^{-\sigma})\cap \mH$ that
\begin{multline*}
  \|\Lambda^{-\sigma} f\|_2= \sup_{g\in \dom((\Lambda^*)^\sigma), \|(\Lambda^*)^\sigma g\|_2=1}
  (\Lambda^{-\sigma} f,(\Lambda^*)^\sigma g) \\
  \approx \sup_{g\in \dom(|D|^\sigma), \||D|^\sigma g\|_2=1} (f,g)=
  \|f\|_{\dot H^{-\sigma}},
\end{multline*}
since $\||D|^\sigma g\|_2= \|g\|_{\dot H^\sigma}$ for $g\in \mH$ and $|D|^\sigma g=0$ for $g\in \mH^\perp$. 
\end{proof}

\section{Applications to the Neumann and Dirichlet problem}   \label{sec:bvp}

It is important to note that in the previous sections, we have always worked with the quantity
$$
  f= \nabla_A u= \begin{bmatrix} \pd_{\nu_A} u \\ \nabla_\ta u \end{bmatrix},
$$
the conormal gradient of a solution $u$, as a whole. 
On the contrary, for the Neumann and Dirichlet problems we need to work with the two 
components $\pd_{\nu_A}u$, the Neumann datum, and $\nabla_\ta u$, the Dirichlet datum, separately.
In doing so, we leave the functional calculus of the bisectorial operator $DB_0$ and go beyond 
the Cauchy integral.

Consider the function spaces
$$
  L_2^\alpha(\R^{1+n}_+):= \begin{cases} L_2(\R^{1+n}_+, t^{\alpha}), & \alpha\in (-1,1], \\
  N_{2,2}(\R^{1+n}_+), & \alpha= -1. \end{cases}
$$
From Theorem~\ref{thm:main} and Proposition~\ref{prop:sob}, it follows that
there is a well defined and bounded trace map 
$$
  \nabla_A u(t,x)\mapsto \nabla_A u(0,x)
$$
taking solutions $u$ of $\divv_{t,x} A(t,x) \nabla_{t,x} u=0$ with  $\nabla_A u\in L_2^\alpha(\R^{1+n}_+)$, to $\nabla_A u|_{\R^n}\in \dot H^{-\sigma}(\R^n)$.
This is true for any bounded accretive coefficients $A$ when $|\alpha|<1$, and in the endpoint cases
$\alpha=\pm 1$ we need to impose the Carleson--Dahlberg condition $\|A(t,x)-A(0,x)\|_*<\infty$.
If furthermore we assume smallness of $\|A(t,x)-A(0,x)\|_\infty$ in the case $|\alpha|<1$, or 
smallness of $\|A(t,x)-A(0,x)\|_*$ in the case $\alpha=\pm 1$, then 
we have equivalence of norms
$$
  \|\nabla_A u|_{\R^n}\|_{\dot H^{-\sigma}(\R^n)}\approx
  \|\nabla_A u\|_{L_2^\alpha(\R^{1+n}_+)}.
$$
In this case, the Hardy type subspace 
$$
  E_A^+ \dot H^{-\sigma}(\R^n):= \{\nabla_A u|_{\R^n}\}\subset \dot H^{-\sigma}(\R^n)
$$
is a closed subspace of $\dot H^{-\sigma}(\R^n)$.
We make the following definitions.

\begin{defn}   \label{defn:wp}
  Let $\alpha\in [-1,1]$ and $\sigma= (\alpha+1)/2\in [0,1]$. 
  Consider bounded and accretive coefficients $A(t,x)$, and assume smallness of $A(t,x)-A(0,x)$
  in the above sense depending on $\alpha$.

  Let $WP(Neu, \dot H^{-\sigma})$ denote the set of coefficients $A$ 
  for which the Neumann map $E_A^+ \dot H^{-\sigma}(\R^n;\C^{(1+n)m})\to \dot H^{-\sigma}(\R^n; \C^m):\nabla_A u|_{\R^n}\mapsto (\nabla_A u|_{\R^n})_\no$ is an isomorphism, so that in particular it has lower bounds
$$
  \|\nabla_A u\|_{L_2^\alpha(\R^{1+n}_+)} \approx \|\nabla_A u|_{\R^n}\|_{\dot H^{-\sigma}(\R^n)}
   \lesssim
   \|\pd_{\nu_A} u|_{\R^n}\|_{\dot H^{-\sigma}(\R^n)}.
$$

  Let $WP(Dir, \dot H^{1-\sigma})$ denote the set of coefficients $A$ 
  for which the Dirichlet map $E_A^+ \dot H^{-\sigma}(\R^n;\C^{(1+n)m})\to \dot H^{-\sigma}(\R^n; \C^{nm}):\nabla_A u|_{\R^n}\mapsto (\nabla_A u|_{\R^n})_\ta$ is an isomorphism, so that in particular it has lower bounds
$$
  \|\nabla_A u\|_{L_2^\alpha(\R^{1+n}_+)} \approx \|\nabla_A u|_{\R^n}\|_{\dot H^{-\sigma}(\R^n)} \lesssim
   \|\nabla_\ta u |_{\R^n}\|_{\dot H^{-\sigma}(\R^n)}.
$$
\end{defn} 

When $A=A(x)$ are $t$-independent, the Hardy subspace $E_A^+ \dot H^{-\sigma}(\R^n)$ is the range
of projection $E_0^+= \chi_+(DB_0)$, which acts boundedly in $\dot H^{-\sigma}(\R^n)$.
We have that
$$
L_\infty\ni B_0\mapsto E_0^+\in\mL(\dot H^{-\sigma}(\R^n))
$$ 
is locally Lipschitz continuous.
More generally for $t$-dependent coefficients, we see from Theorem~\ref{thm:main} that
$$
   E_A^+ g= E_0^+g +\int_0^\infty \Lambda e^{-s\Lambda} E_0^-\E_s \big( (I-S\E)^{-1} e^{-t\Lambda} E_0^+ g\big)_s ds
$$
is a projection onto the Hardy space $E_A^+ \dot H^{-\sigma}(\R^n)$ along 
the null space $E_0^- \dot H^{-\sigma}(\R^n)$.
In this case we have that
$$
  \|E_A^+-E_0^+\|_{ \dot H^{-\sigma}(\R^n)\to  \dot H^{-\sigma}(\R^n)}\lesssim 
  \begin{cases}
     \|\E\|_{L_\infty(\R^{1+n}_+)}, & |\alpha|<1, \\
     \|\E\|_*, & \alpha=\pm 1.
  \end{cases}
$$
Since in this way, the Hardy space of solutions depends continuously on the coefficients, 
we obtain the following perturbation result.

\begin{prop}
If we have a well posed boundary value problem for $t$-independent coefficients $A= A(x)\in WP(Neu, \dot H^{-\sigma})$,
then there exists $\epsilon>0$ such that $\tilde A\in WP(Neu, \dot H^{-\sigma})$ whenever
$\sup_{(t,x)\in \R^{1+n}_+}|\tilde A(t,x)-A(x)|<\epsilon$ when $|\alpha|<1$, and
whenever $\sup_{x\in \R^n} |\tilde A(0,x)-A(x)|<\epsilon$ and
$\|\tilde A(t,x)-\tilde A(0,x)\|_{*}<\epsilon$ when $\alpha=\pm 1$.

The corresponding result also holds for the Dirichlet problem.
\end{prop}

\begin{ex}
  The optimal case is when $\alpha=0$, in which case all coefficients belong to $WP(Neu, \dot H^{-1/2})$ and $WP(Dir, \dot H^{1/2})$.
  This is a simple consequence of Gauss' theorem, which yields
\begin{multline*}
  \iint_{\R^{1+n}_+} (A(t,x)\nabla_{t,x} u, \nabla_{t,x}u) \eta_\epsilon(t) dtdx \\=
  2\epsilon\int_{1/(2\epsilon)}^{1/\epsilon}\int_{\R^n} (\pd_{\nu_A}u_t, u_t) dxdt
   -\epsilon^{-1}\int_\epsilon^{2\epsilon}\int_{\R^n} (\pd_{\nu_A}u_t, u_t) dxdt,
\end{multline*}
with $\eta_\epsilon(t)$ as in the proof of Theorem~\ref{thm:Sest}.
Taking limits $\epsilon\to 0$, by accretivity of $A$ this gives the estimate
$$
  \int_0^\infty \|f_t\|^2_2 dt \lesssim \|(f_0)_\no\|_{\dot H^{-1/2}}  \|(f_0)_\ta\|_{\dot H^{-1/2}} 
$$
of the conormal gradient $f$.
Since $\max( \|(f_0)_\no\|_{\dot H^{-1/2}}  \|(f_0)_\ta\|_{\dot H^{-1/2}} )\approx \|f_0\|_{\dot H^{-1/2}}
\approx  \|f\|_{L_2(\R^{1+n}_+)}$, we can absorb either factor of the right hand side, on the left, 
and obtain the claimed lower bounds.
\end{ex}

\begin{ex}
Much more subtle are the endpoint cases $\alpha=\pm 1$.
The Neumann problem with data in $L_2(\R^n)$ ($\alpha=-1$) is usually denoted $(N)_2$ in the literature,
and the Dirichlet problem with data in $\dot H^1(\R^n)$ ($\alpha=-1$) is usually denoted $(R)_2$
and referred to as the regularity problem. See Kenig~\cite[Sec. 1.8]{Kenig}.

The Dirichlet problem with data in $L_2(\R^n)$ ($\alpha=+1$) is usually denoted $(D)_2$.
In this case our Definition~\ref{defn:wp} differs from the standard one in that we require the 
square function estimate
$$
\iint_{\R^{1+n}_+} |\nabla u|^2 tdtdx\lesssim \int_{\R^n}|u|^2 dx
$$
rather than the non-tangential maximal estimate which usually defines $(D)_2$. See Kenig~\cite[Sec. 1.8]{Kenig}.
It should be clear from Section~\ref{sec:ops} why we prefer to define the Dirichlet problem 
through the square function estimate here.

Note that by maximal function estimates proved in \cite[Thm. 2.4]{AA1}, $A\in WP(Dir, L_2)$ 
implies that $(D)_2$ holds, modulo one subtle point.
There are coefficients 
$A\in WP(Dir, L_2)$ where even for good Dirichlet data $\phi\in L_2(\R^n)\cap \dot H^{1/2}(\R^n)$, the solutions $u^0\in L_2(\R^{1+n}_+,t)$ and $u^{1/2}\in L_2(\R^{1+n}_+)$ are distinct.
See \cite[Sec. 5]{AAM} and \cite{AxNon}.
In this case, there will exist some Sobolev space $0<\sigma <1/2$ where $A\notin WP(Dir, \dot H^\sigma)$.
Here $u^{1/2}$ is seen to be the solution obtained from Lax--Milgram's theorem, which is the one for which non-tangential maximal estimates are required in the problem $(D)_2$.
However, the solution which is estimated in \cite[Thm. 2.4]{AA1} is $u^0$.

There are coefficients $A$ for which these endpoint 
boundary value problems are not well posed.
For positive results, it is known that $WP(Neu, L_2)$, $WP(Dir, \dot H^1)$ and $WP(Dir, L_2)$, as well as the less well known 
fourth end point boundary value problem $WP(Neu, \dot H^{-1})$, contain all $t$-independent 
coefficients $A(x)$ which are Hermitean, $A^*(x)= A(x)$, or of block form, $A= \begin{bmatrix} a & 0 \\ 0 & d \end{bmatrix}$, or are constant $A(x)=A_0$.
\end{ex}

There is also a duality result for these boundary value problem, which reads as follows.
\begin{prop}   \label{prop:duality}
  For the Dirichlet problem, we have 
$$
  A\in WP(Dir, \dot H^{\sigma}) \qquad\text{if and only if}\qquad A^*\in WP(Dir, \dot H^{1-\sigma}).
$$
  For the Neumann problem, we have 
$$
  A\in WP(Neu, \dot H^{-\sigma}) \qquad\text{if and only if}\qquad A^*\in WP(Neu, \dot H^{\sigma-1}).
$$
\end{prop}

This can be proved as in \cite[Sec. 17.2]{AR2}.
We remark that it is well known that the Dirichlet problem $(D)_2$ holds whenever the regularity problem
$(R)_2$ holds, whereas the reverse implication is not true in general.
The reason that this reverse implication holds in Proposition~\ref{prop:duality}, is that we require 
a stronger square function estimate rather than a non-tangential maximal estimate for the Dirichlet problem
with data in $L_2(\R^n)$.

\bibliographystyle{acm}

\end{document}